\documentclass[12pt]{amsart}
\usepackage{geometry}   
\usepackage[colorlinks,citecolor = red, linkcolor=blue,hyperindex]{hyperref}
\usepackage{euscript,eufrak,verbatim, mathrsfs}
\usepackage[psamsfonts]{amssymb}
\usepackage[usenames]{color}
\usepackage{bbm}
\usepackage{graphicx}
 \usepackage{color}
 \usepackage{float}

 \usepackage{euscript}
 \usepackage{xypic}
\usepackage{helvet}         
\usepackage{courier}        
\usepackage{type1cm}        
\usepackage{multicol}        
\usepackage[bottom]{footmisc}

\newtheorem{theorem}{Theorem}[section]
\newtheorem*{theorem*}{Theorem B}
\newtheorem{lemma}[theorem]{Lemma}

\newtheorem{proposition}[theorem]{Proposition}
\newtheorem{corollary}[theorem]{Corollary}
\newtheorem{definition}[theorem]{Definition}
\newtheorem*{definition*}{Definition}
\newtheorem*{remark*}{Remark}

\newtheorem*{observation*}{Observation}

\newtheorem*{assumption*}{Assumption}
\newtheorem*{question*}{Question}
\newtheorem{remark}[theorem]{Remark}

\newcommand{\R}{\mathbb{R}}
\newcommand{\N}{\mathbb{N}}
\newcommand{\Z}{\mathbb{Z}}

\newcommand{\C}{\mathbb{C}}
\newcommand{\E}{\mathbb{E}}

\newcommand{\T}{\mathbb{T}}

\newcommand{\tr}{\mathrm{tr}}

\newcommand{\dt}{\mathrm{det}}

\newcommand{\an}{\text{\, and \,}}

\begin{document}

\title[$\psi$-mixing and $L^q$-dimensions of DPP]{Stationary determinantal processes: $\psi$-mixing property and $L^q$-dimensions}

\author
{Shilei Fan}
\address
{Shilei Fan: School of Mathematics and Statistics, Hubei Key Laboratory of Mathematical Sciences, Central  China Normal University,  Wuhan, 430079, China}
\email{slfan@mail.ccnu.edu.cn}

\author
{Lingmin Liao}
\address
{Lingmin Liao: LAMA, UMR 8050, CNRS,
Universit\'e Paris-Est Cr\'eteil Val de Marne, 61 Avenue du
G\'en\'eral de Gaulle, 94010 Cr\'eteil Cedex, France}
\email{lingmin.liao@u-pec.fr}

\author
{Yanqi Qiu}
\address
{Yanqi Qiu: Institute of Mathematics, Academy of Mathematics and Systems Science, Chinese Academy of Sciences, Beijing 100190, China}
\email{yanqi.qiu@hotmail.com}

\begin{abstract}
The results of this paper are 3-folded. Firstly, for  any stationary determinantal process on the integer lattice, induced by strictly positive and strictly contractive involution kernel, we obtain the necessary and sufficient condition for the $\psi$-mixing property. Secondly,  we obtain the existence of the $L^q$-dimensions of the stationary determinantal measure on symbolic space $\{0, 1\}^\N$ under appropriate conditions. Thirdly, the previous two results together imply the precise increasing rate of  the longest common substring of a typical pair of points in $\{0, 1\}^\N$.
\end{abstract}

\subjclass[2010]{Primary 60G55, 37A25; Secondary 60G10.}
\keywords{stationary determinantal point processes; $\psi$-mixing, $L^q$-dimensions}

\maketitle

\setcounter{equation}{0}

\section{Introduction}

\subsection{Stationary determinantal point processes}

Let $\T: = \R/\Z$ be the unit circle.   For any  Borel function $f: \T  \rightarrow [0, 1]$, the convolution kernel $K_f : \Z\times \Z \to \C$   defined by
\begin{align}\label{def-kernel}
K_f (n, m) : = \widehat{f}(n-m) = \int_{0}^1  f(t) e^{- 2 \pi i  (n-m) t  } dt, \quad \forall n, m \in \Z
\end{align}
defines a self-adjoint positive contractive operator on the Hilbert space $\ell^2(\Z)$ and thus, by Macchi-Soshnikov-Shirai-Takahashi Theorem \cite{Lyons-DPP-IHES,Lyons-stationary,ST-palm2}, induces a stationary determinantal point process, denoted by $\mu_f$,  on $\Z$. More precisely, $\mu_f$ is the probability measure on $\{0, 1\}^\Z$ such that for any distinct points $n_1, \cdots, n_k \in \Z$, we have
\[
\mu_f \Big( \Big\{x \in \{0, 1\}^\Z\Big| x_{n_1} = x_{n_2} = \cdots = x_{n_k} =1 \Big\}\Big) =  \det (\widehat{f}(n_i- n_j))_{1\le i, j \le k }.
\]
Similarly, the kernel \eqref{def-kernel} restricted on $\N\times \N$ (where $\N$ is the set of  non-negative integers), induces a determinantal probability measure,  denoted by $\mu_f^{+}$,  on $\{0, 1\}^\N$, such that for any distinct points $n_1, \cdots, n_k \in \N$, we have
\[
\mu_f^{+} \Big( \Big\{x \in \{0, 1\}^\N\Big| x_{n_1} = x_{n_2} = \cdots = x_{n_k} =1 \Big\}\Big) =  \det (\widehat{f}(n_i- n_j))_{1\le i, j \le k }.
\]

 Clearly, $\mu_f$ is invariant under the shift operator $\sigma: \{0, 1\}^\Z\rightarrow \{0, 1\}^\Z$ defined by \[
\sigma ((x_n)_{n\in \Z})=(x_{n+1})_{n\in \mathbb{Z}}, \quad x   = (x_n)_{n\in \Z} \in \{0, 1\}^\Z.
\]
 Similarly, $\mu_f^{+}$ is invariant under the one-sided shift on $\{0, 1\}^\N$, which by slightly abusing the notation will also be denoted by $\sigma$.

In this paper, we  will consider the probability-measure-preserving dynamical systems: $(\{0, 1\}^\Z, \sigma, \mu_f)$ and  $(\{0, 1\}^\N, \sigma, \mu_f^{+})$ and will study  the  mixing properties, the  fractal dimensions etc.  The reader is referred to  Lyons and Steif \cite[Corollary 8.4]{Lyons-stationary} for more properties on the stationary determinantal point processes on $\Z^d$.

\subsection{The $\psi$-mixing property}
Let us recall the $\psi$-mixing property for stochastic processes on $\Z$.  In what follows,  let $(\xi_n)_{n\in \Z}$  be the  sequence  of  random variables  taking values  in $\{0, 1\}$ with joint distribution $\mu_f$, that is, for any distinct points  $n_1, \cdots, n_k \in \Z$,
\[
\E(\xi_{n_1} \cdots \xi_{n_k})=  \det (\widehat{f}(n_i - n_j))_{1\le i, j \le k}.
\]
For any pair of integers $n\leq m$, let
\begin{align}\label{def-sigma-F}
\mathcal{F}_{n}^m:=\vee(\xi_n,\xi_{n+1},\cdots,\xi_m)
\end{align}
 be the
sigma-algebra generated by the random variables $\xi_n, \cdots, \xi_m$. Similarly, set
\[\mathcal{F}_{-\infty}^{n}=\vee(\cdots,\xi_{n-1},\xi_{n}), \quad   \mathcal{F}_{n}^{+\infty}=\vee(\xi_n,\xi_{n+1},\cdots).
\]
The $\psi$-function of the measure $\mu_f$ is defined as follows: for each integer $\ell\geq 1$, set
\begin{align}\label{def-psi-n}
\psi_{\mu_f}(\ell): =\sup_{\substack{A\in \mathcal{F}_{-\infty}^{0},B\in\mathcal{F}^{+\infty}_{\ell}\\ \mu_f(A)>0, \mu_{f}(B)>0}}\left|\frac{\mu_f(A\cap B)}{
\mu_f(A)\mu_f(B)}-1\right| \in [0, \infty].
\end{align}

\begin{definition}
We say that $\mu_f$  is $\psi$-mixing if
$
\lim_{\ell \to +\infty}\psi_{\mu_f}(\ell)=0.
$
\end{definition}

Recall that we say that an integrable function  $f: \T\rightarrow \C$ is  in the Sobolev space $H^{1/2}(\T)$ if
\[
\sum_{n=-\infty}^{+\infty}|n| \cdot |\widehat{f}(n)|^{2}<  \infty.
\]

\begin{theorem}\label{psi-mixing}
Let $f: \T \to [0,1]$ be an integrable function  such that $f \not \equiv 0$ and $f \not \equiv 1$..  Then the  $\psi$-function  of the measure $\mu_f$ satisfies: for any integer $\ell \ge 1$,
\begin{align}\label{lower-bdd}
\psi_{\mu_f}(\ell)\geq 1-\exp\Big(- \frac{1}{\ell + 1}\sum_{n=\ell + 1}^\infty |n||\widehat{f}(n)|^{2}\Big).
\end{align}
In particular,  if $\mu_f$ is $\psi$-mixing, then  $f\in H^{1/2}(\T)$. Conversely, assume that there exists $\tau>0$ such that $f$ satisfies:
\begin{align}\label{suff-cond}
\text{$f\in H^{1/2}(\T)$ and $\tau\leq f\leq 1-\tau$.}
\end{align}
Then $\mu_f$  is $\psi$-mixing and  its  $\psi$-function satisfies
\begin{align}\label{upper-bound}
\psi_{\mu_f}(\ell)&\leq \frac{1}{\tau^2} \Big(\sum_{n=\ell+1}^{+\infty}|n||\widehat{f}(n)|^{2}\Big)\cdot \exp\Big(1+\frac{1}{\tau^2}\sum_{n=\ell+1}^\infty |n||\widehat{f}(n)|^{2}\Big).
\end{align}
\end{theorem}

\begin{remark}
Shirai and Takahashi \cite{ST-palm2} implicitly proved that the condition \eqref{suff-cond} is sufficient for the $\psi$-mixing property of $\mu_f$. In full generality, it is unclear to the authors when a sationary determinantal point process is $\psi$-mixing.
\end{remark}

\subsection{Correlation dimensions}

Let $(X,d)$ be a metric space equipped with a Borel  probability  measure $\mu$.    For any real number $q>1$, we define the lower and upper $L^q$-dimensions respectively by the formulae:
\begin{align}\label{def-q-dim}
\begin{split}
\underline{\dim}_q\mu:&=\frac{1}{q-1}\liminf_{r\to 0}{\log \int_X \mu(B(x,r))^{q-1} d\mu(x) \over \log r},
\\
\overline{\dim}_q\mu: & =\frac{1}{q-1}\limsup_{r\to 0}{\log \int_X \mu(B(x,r))^{q-1} d\mu(x) \over \log r}.
\end{split}
\end{align}
If  $\underline{\dim}_q\mu=\overline{\dim}_q\mu$, then the common value, denoted by $\dim_q\mu$, is called the {\it $L^q$-dimension} of $\mu$. When $q = 2$,  the quantity $\dim_2\mu$ is also called the  {\it correlation dimension}  of $\mu$.  The relation between the $L^q$-dimensions and other dimensions, such as Hausdorff dimension, of a measure were investigated by Fan, Lau and Rao \cite{FLR-dim}. For more details concerning  correlation dimensions, see Pesin's Book \cite[Chapter 6]{Pesin-book97}.

We  will study the $L^p$-dimension of the stationary determinantal measures $\mu_f^{+}$ on the metric space $(\{0, 1\}^\N, d)$, where $d$ is defined by
\begin{align}\label{def-metric}
d(x, y) = 2^{- \min \{ n\ge 0 | x_n \ne y_n \}}, \quad \forall x, y \in \{0, 1\}^\N.
\end{align}
Since the possible values of the metric \eqref{def-metric} are of the form $2^{-N}$, for any $q > 1$,  we define
\begin{align}\label{def-S-N}
S_N^{(q)}(\mu_f^{+}):=\int_{\{0, 1\}^\N} \mu_f^{+}(B(x,2^{-N}))^{q-1} d\mu_f^{+}(x).
\end{align}

\begin{proposition}\label{prop-sub-m}
Let $f: \T \to [0,1]$ be a Borel function such that either $f\leq \frac{1}{2}$ or $f\geq \frac{1}{2}$. Then for any integer  $q\in \N$ with $q \ge 2$, the function
\[
\N \ni N \mapsto  S_N^{(q)}(\mu_f^{+})
\]
is sub-multiplicative, that is, for any integers $M, N\ge 1$, we have
\[
 S_{M + N}^{(q)} (\mu_f^{+}) \le    S_{M}^{(q)} (\mu_f^{+}) \cdot  S_{N}^{(q)}(\mu_f^{+}).
\]
\end{proposition}

As a corollary of Proposition \ref{prop-sub-m}, we have
\begin{theorem}\label{thm-CD}
Let $f: \T \to [0,1]$ be a Borel function such that either $f\leq \frac{1}{2}$ or $f\geq \frac{1}{2}$. Then for any integer  $q\in \N$ with $q \ge 2$,  the $L^q$-dimension $\dim_q\mu_f^{+}$ of the determinantal measure $\mu_f^{+}$ on the metric space $(\{0, 1\}^\N, d)$ exists.
\end{theorem}

\begin{remark}
We conjecture that Theorem \ref{thm-CD} holds in full generality: that is, it holds for all $f: \T\rightarrow [0, 1]$ and all  real numbers $q > 1$. In a forthcoming paper,  we show that for a real analytic function $f$ such that $\tau \le f \le 1- \tau$ for some $\tau \in (0, 1/2)$, the $L^q$-dimensions of $\mu_f$ exist for all  real numbers $q>1$.
\end{remark}
We have the following upper and lower estimates for the correlation dimensions.
\begin{proposition}
Let $f: \T \to [0,1]$ be a Borel function. Then we have
\[\underline{\dim}_2\mu_f^{+}\geq  \frac{1}{ \log 2} \int_{0}^{1}\log \left( \frac{2}{1+(   2f(e^{i 2 \pi t}) -1 )^2}\right) dt,\]
and
\[\overline{\dim}_2\mu_f^{+}\leq 1- \frac{1}{ \log 2} \int_{0}^{1}\log \left( \frac{[1+\beta (2f(e^{i 2 \pi t}) -1)]^2}{1+\beta^2}\right) d t,  \quad \forall \beta \in [-1,1]. \]
\end{proposition}

\section{Preliminaries}

\subsection{Equalities and inequalities of determinants}

The following elementary results on determinants will be useful for us.
\begin{itemize}
\item The Sylvester's determinant identity: if $A$ and $B$ are two matrices such that both matrix-products $AB$ and $BA$ can be defined, then
\begin{align}\label{syl-id}
\det(1-AB)=\det(1-BA).
\end{align}
\item The Fischer's inequality: write a non-negative definite matrix $M$ in the following block form
\[
M = \left[
\begin{array}{cc}
A & B
\\
B^* & C
\end{array}
\right], \quad \text{$A$ and $C$ are square matrices,}
\]
then
\begin{align}\label{Fischer-inq}
\det(M) \le \det (A) \det (C).
\end{align}
\item For any positive integer $N\ge 1$, we denote
\begin{align}\label{def-set-N}
[N]: = \{1, 2, \cdots, N\}.
\end{align}
 For an $N\times N$ matrix $L$ and a subset $J = \{j_1, \cdots, j_n\} \subset [N]$, we denote by $L_J$ the submatrix of $L$ with $j_1, \cdots, j_n$-th rows and columns.  Then the determinant $\det (1 + L)$ can be expanded (cf. e.g., \cite[formulae (2.6) and (2.7)]{ST-palm2}) as
\begin{align}\label{det1+L}
\det(1+L)=\sum_{J\subset[N]}\det L_{J},
\end{align}
where we used the convention $\det L_\emptyset := 1$.
\item Let $A$ be a square matrix such that both $A$ and $1 - A$ non-negative definite, then
\begin{align}\label{det-exp}
\det(1-A)\leq \exp(-\tr(A)).
\end{align}
\end{itemize}

\subsection{Operator-order and operator ideals}

Let $\mathcal{H}$  be a Hilbert space and  let $B(\mathcal{H})$ be the set of bounded linear operators on $\mathcal{H}$. Recall that the operator-order for Hermitian operators is defined as follows: For any two Hermitian operators $A, B \in B(\mathcal{H})$, we write $A\le B$ if $B-A$ is a non-negative operator. Clearly, for any two Hermitian operators $A$ and $B$ with $A\leq B $, we have
  \begin{align}\label{op-geq}
C^*AC\leq C^*BC, \quad \forall C\in B(\mathcal{H}).
\end{align}
For any $A\in B(\mathcal{H})$, we denote by $\| A\|$ its the operator norm. Then  for any self-adjoint $B \in B(\mathcal{H})$, we have
\begin{align}\label{self-bound}
-\|B\|\cdot I \leq B\ \leq \|B\|\cdot I,
\end{align}
where $I$ is the identity operator on $\mathcal{H}$.

We will use some elementary operator ideal inequalities as follows, all the details can be found in Simon \cite[Chapter 1]{Simon-trace}. Denote by $\mathcal{I}_\infty(\mathcal{H})$ the ideal in $B(\mathcal{H})$ consists of compact operators on $\mathcal{H}$. For any $p\ge 1$,  the Schatten-von Neumann operator ideal $\mathcal{I}_p(\mathcal{H})$ is a sub-ideal of $\mathcal{I}_\infty(\mathcal{H})$ consists of operators $A\in \mathcal{I}_\infty(\mathcal{H})$ such that
\begin{align}\label{def-p-norm}
\| A\|_p = \Big(\sum_{n = 1}^\infty [\lambda_n(A^*A)]^{p/2} \Big)^{1/p}<\infty,
\end{align}
where $(\lambda_n(A^*A))_{n\in \N}$ is the sequence (counting multiplicities) of the eigenvalues of the compact operator $A^*A$.The quantity \eqref{def-p-norm} defines a norm on the linear space $\mathcal{I}_p(\mathcal{H})$ and making it a Banach space. The following operator ideal inequalities (for $p =1$ and $p =2$) will be useful for us.  Let $C\in \mathcal{I}_p(\mathcal{H})$ and $ B, B^\prime\in B(\mathcal{H})$. Then
\begin{align}\label{p-norm}
\|BCB^{\prime}\|_p\leq \|B\|\cdot\|C\|_p\cdot\|B^{\prime}\|.
\end{align}
The non-commutative H\"older inequality says that if $\frac{1}{p} = \frac{1}{q} + \frac{1}{r}$ with $p, q, r \ge 1$, then for any $A\in \mathcal{I}_q(\mathcal{H})$ and $B\in \mathcal{I}_r(\mathcal{H})$, we have (cf., Simon \cite[Theorem 2.8]{Simon-trace})
\begin{align}\label{nc-holder}
\| AB\|_p \le \| A\|_q \| B\|_r.
\end{align}
If $A$ is a non-negative operator, then $\| A\|_1 = \tr(A)$.  For $p =2$, the norm $\|A\|_2$ is called the Hilbert-Schmidt norm and will also be denoted as $\| A\|_{HS}$.

The following  elementary results on the non-negative definite matrices will be usefulf for us. For the reader's convenience,  we include  the proofs.
\begin{lemma}\label{1-norm}
Let $A, B$ be two  $N\times N$ Hermitian matrices such that $A$ is non-negative definite and $-A\leq B\leq A$,
then $\|B\|_1\leq  \|A\|_1.$
\end{lemma}
\begin{proof}
By the standard argument of small perturbation, that is, replacing $A$ by  $A + \varepsilon I$ for arbitrarily  small $\varepsilon > 0$, we may assume without loss of generality that $A$ is invertible. By \eqref{op-geq}, the assumption $-A \le B \le A$ is equivalent to $
-I \le A^{-1/2} B A^{-1/2} \le I$,  which is in turn equivalent to the operator-norm inequality $\| A^{-1/2}B A^{-1/2}\| \le 1$.
Therefore, by first applying \eqref{nc-holder} and then applying  \eqref{p-norm},   we get
\begin{align*}
\| B\|_1 &  = \| A^{1/2} (A^{-1/2} B A^{-1/2}) A^{1/2}\|_1  \le \| A^{1/2} (A^{-1/2} B A^{-1/2})\|_2 \cdot \| A^{1/2}\|_2 \\
& \le \| A^{1/2} \|_2 \cdot \| A^{-1/2} B A^{-1/2}\| \cdot \| A^{1/2}\|_2 \le \| A^{1/2}\|_2^2 = \tr(A) = \| A\|_1.
\end{align*}
\end{proof}

\begin{lemma}\label{lem-pd}
Let  $A, C$ be two $N\times N$ square matrices. Assume that $A$ is invertible and non-negative definite.  Then
\[
 \left[
\begin{array}{cc}
A & C
\\
C^* & A
\end{array}
\right] \ge 0 \Longleftrightarrow  A^{-1/2} C^* A^{-1} C A^{-1/2} \le I.
\]
\end{lemma}

\begin{proof}
 The result follows immediately by observing the following equalities
\[
\left[
\begin{array}{cc}
A^{-1/2} & 0
\\
0 & A^{-1/2}
\end{array}
\right]
\left[
\begin{array}{cc}
A & C
\\
C^* & A
\end{array}
\right]     \left[
\begin{array}{cc}
A^{-1/2} & 0
\\
0 & A^{-1/2}
\end{array}
\right]
=
\left[
\begin{array}{cc}
I & A^{-1/2} C A^{-1/2}
\\
A^{-1/2} C^* A^{-1/2} & I
\end{array}
\right]
\]
and by setting $M = A^{-1/2} C A^{-1/2}$, we have
\[
\left[
\begin{array}{cc}
I & 0
\\
-M^* & I
\end{array}
\right]
\left[
\begin{array}{cc}
I & M
\\
M^* & I
\end{array}
\right]
\left[
\begin{array}{cc}
I & - M
\\
0 & I
\end{array}
\right]
= \left[
\begin{array}{cc}
I & 0
\\
0 & I - M^* M
\end{array}
\right].
\]
\end{proof}

\section{The $\psi$-mixing property}
In this section, we shall investigate the $\psi$-mixing property of the  stationary determinantal measure $\mu_f$.

\subsection{Notation}

 In what follows, we shall use the following notation: let $N\ge 1$ be a positive integer.
\begin{itemize}
\item For any finite word $\epsilon= \epsilon_1\epsilon_2\cdots\epsilon_{N} \in \{0,1\}^N$, define the $N$-cylinder set
\begin{align}\label{def-cyl}
[\epsilon]:=\Big\{ x   = (x_k)_{k\in \Z}\in \{0,1\}^{\Z}\Big| x_1=\epsilon_1, x_{2}=\epsilon_2, \cdots, x_{N}=\epsilon_{N}\Big\}.
\end{align}
\item For any subset $S\subset \{0, 1\}^N$, let $A_S\subset \{0, 1\}^\Z$ be the subset defined by
\begin{align}\label{def-A-S}
A_S: = \bigsqcup_{\epsilon \in S}[\epsilon].
\end{align}
\item For any vector $v \in \C^N$, we write $D_N(v)$ the diagonal matrix with diagonal entries the coordiantes of $v$. For instance,  we will use the notation: for any $\epsilon, \epsilon' \in \{0, 1\}^N$, we write $D_{2N}(\epsilon, \epsilon')$ for the diagonal matrix
\[
D_{2N} (\epsilon, \epsilon') =
\left[
\begin{array}{cc}
D_N(\epsilon) & 0
\\
0 & D_N(\epsilon')
\end{array}
\right]
\]
\item For a bounded measurable function  $\phi: \T\rightarrow \C$,  let $T(\phi)$ be the bounded operator on $\ell^2(\Z)$ corresponding to the bi-infinite  Toeplitz matrix
\begin{align}\label{def-T-phi}
T(\phi)=\Big[ \widehat{\phi}(i-j) \Big]_{i, j \in \Z}.
\end{align}
 For any subset  $J\subset \Z$, let $T_{J}(\phi)$ be the operator on $\ell^2(J)$ corresponding to the matrix
\begin{align}\label{def-T-J}
T_J(\phi)=\Big[\widehat{\phi}(i-j)\Big]_{i, j \in J}.
\end{align}
In particular, if $J = [N] = \{1, \cdots, N\}$,  the Toeplitz matrix $T_J(\phi)$ will be denoted by $T_N(\phi)$, that is,
\[
T_N(\phi) = \Big[\widehat{f}(i-j)\Big]_{1\le i, j \le N}.
\]
\item For any $\ell\ge 1$, we define a square matrix  $\Lambda_{N, \ell}(\phi)$  by
\begin{align}\label{def-lambda-N}
\Lambda_{N, \ell}(\phi) = \left[ \widehat{\phi} \Big( i - (j + N+\ell)\Big) \right]_{1\le i, j \le N}.
\end{align}
\end{itemize}

\subsection{Proof of $\psi$-mixing property}

Fix an integrable function $f: \T \rightarrow [0, 1]$ such that $f\not\equiv 0 $ and $f\not\equiv 1$. Let $\mu_f$ be  the  stationary determinantal measure on $\{0, 1\}^\Z$.

Recall the definition \eqref{def-sigma-F} of the sigma-algebras $\mathcal{F}_n^m$ and the definition \eqref{def-psi-n} for $\psi_{\mu_f}$.
\begin{lemma}[{\cite[Theorem 4.2]{Lyons-stationary}}]\label{lem-Lyons}
The assumption  $f\not\equiv 0 $ and $f\not\equiv 1$ implies that $\mu_f(A)>0$ for any non-empty set   $A\in \mathcal{F}_{m}^{m+\ell}$.
In particular, for any integer $N\ge 1$, the matrix $T_N(f)$ is invertible.
\end{lemma}

Clearly, since the measure $\mu_f$ is invariant under the shift operator  $\sigma$  on $\{0, 1\}^\Z$, by Lemma \ref{lem-Lyons}  and approximation, we have the following simple lemma.

\begin{lemma}\label{lem-approx}
For any integer $\ell \ge 1$, we have
\[
\psi_{\mu_f}(\ell)=\sup_{N\geq 1}\sup_{A, B\in \mathcal{F}_{1}^{N} }\left|\frac{\mu_f(A\cap \sigma^{-(N+\ell)} B)}{
\mu_f(A)\mu_f(B)}-1\right|.
\]
\end{lemma}

The following lemma is classical, we include its proof for completeness.
\begin{lemma}\label{lem-sup}
For any integer $\ell \ge 1$, we have
\begin{align}\label{psi-f-sup}
\psi_{\mu_f}(\ell) = \sup_{N\ge 1}\sup_{\epsilon, \epsilon'\in \{0, 1\}^N}\left| \frac{ \mu_f([\epsilon] \cap \sigma^{-(N+\ell)}[\epsilon']) }{\mu_f([\epsilon]) \mu_f([\epsilon'])}  - 1   \right|.
\end{align}
\end{lemma}

\begin{proof}
Clearly, by Lemma \ref{lem-approx}, we only need to show that the LHS of \eqref{psi-f-sup} is not greater than the RHS of \eqref{psi-f-sup}.

Recall the definition  \eqref{def-A-S} of the subset $A_S\subset \{0, 1\}^\Z$. Clearly, all subsets in $\mathcal{F}_1^N$ are of the form $A_S$ for some subset $S\subset \{0, 1\}^N$.  Therefore,  in view of Lemma \ref{lem-approx}, we shall calculate  the following quantity:  for any integers $N\ge 1, \ell \ge 1$ and any pair $S, S' \subset \{0, 1\}^N$,
\begin{align}\label{def-psi-S-S}
\psi_{N}^{S,S^{\prime}}(\ell): =\left|\frac{\mu_f(A_S\cap \sigma^{-(N+\ell)} A_{S^{\prime}})}{
\mu_f(A_S)\mu_f(A_{S^{\prime}})}-1\right|.
\end{align}
By the definition for the subsets $A_S$ and $A_{S'}$, we have
\[
\mu_f(A_S\cap \sigma^{-(N+\ell)} A_{S^{\prime}})  = \sum_{\epsilon \in S,  \epsilon' \in S'}  \mu_f([\epsilon] \cap \sigma^{-(N+\ell)}[\epsilon']), \quad \mu_f(A_S)  = \sum_{\epsilon \in S}  \mu_f([\epsilon]).
\]
Therefore, we obtain
\begin{align*}
\psi_{N}^{S,S^{\prime}}(\ell): & =\frac{1}{
\mu_f(A_S)\mu_f(A_{S^{\prime}})}\left|  \sum_{\epsilon \in S,  \epsilon' \in S'}  \Big(  \mu_f([\epsilon] \cap \sigma^{-(N+\ell)}[\epsilon'])   - \mu_f([\epsilon]) \mu_f([\epsilon']) \Big)  \right|
\\
 & \le  \frac{1}{
\mu_f(A_S)\mu_f(A_{S^{\prime}})} \sum_{\epsilon \in S,  \epsilon' \in S'}   \mu_f([\epsilon]) \mu_f([\epsilon']) \left| \frac{ \mu_f([\epsilon] \cap \sigma^{-(N+\ell)}[\epsilon']) }{\mu_f([\epsilon]) \mu_f([\epsilon'])}  - 1   \right|
\\
& \le \frac{1}{
\mu_f(A_S)\mu_f(A_{S^{\prime}})}  \left( \sum_{\epsilon \in S,  \epsilon' \in S'}   \mu_f([\epsilon]) \mu_f([\epsilon']) \right) \cdot \sup_{\epsilon, \epsilon'\in \{0, 1\}^N}\left| \frac{ \mu_f([\epsilon] \cap \sigma^{-(N+\ell)}[\epsilon']) }{\mu_f([\epsilon]) \mu_f([\epsilon'])}  - 1   \right|
\\
& = \sup_{\epsilon, \epsilon'\in \{0, 1\}^N}\left| \frac{ \mu_f([\epsilon] \cap \sigma^{-(N+\ell)}[\epsilon']) }{\mu_f([\epsilon]) \mu_f([\epsilon'])}  - 1   \right|.
\end{align*}
This completes the proof of the lemma.
\end{proof}

We now focus on estimation of the following ratio:
\begin{align}\label{def-R-N}
R_{N,\ell}^{f}(\epsilon,\epsilon^{\prime}):=\frac{\mu_f([\epsilon]\cap  \sigma^{-(N+\ell)}[\epsilon^{\prime}])}{\mu_f([\epsilon])\mu_f([\epsilon^{\prime}])}, \quad \epsilon, \epsilon' \in \{0, 1\}^N.
\end{align}

We shall need the explicit formula for the weight of a determinantal measure on the cylinder set.  The equalities in Proposition \eqref{prop-ST} are just another way of writing for a formula in Shirai-Takahashi \cite[formula (2.1)]{ST-palm2}.

\begin{proposition}[{\cite[formula (2.1)]{ST-palm2}}]\label{prop-ST}
Fix an integer $N\ge 1$. Then for any $\epsilon \in \{0,1\}^{N}$,
\begin{align}\label{DPP-f-C}
\mu_f([\epsilon])=\det\Big(D_{N}(2\epsilon-1)T_{N}(f)+D_{N}(1-\epsilon)\Big)
\end{align}
and  for any $\epsilon,\epsilon^{\prime}\in \{0,1\}^{N}$ and any integer $\ell \ge 1$,
\begin{align}\label{DPP-f-C-bis}
\mu_{f}\Big([\epsilon]\cap  \sigma^{-(N+\ell)} [\epsilon']\Big)=
\det\Big(D_{2N}(2\epsilon-1,2\epsilon^{\prime}-1)T_{J_{N, \ell}}(f)+D_{2N}(1-\epsilon,1-\epsilon^{\prime})\Big),
\end{align}
where $J_{N, \ell}$ is  the ordered set of cardinality $2N$ defined by
\[
J_{N, \ell} = \{1, \cdots, N\} \cup \{ N + \ell + 1, \cdots, N+ \ell + N\}.
\]
In particular, using the notation \eqref{def-T-J} and \eqref{def-lambda-N}, the matrix  $T_{J_{N, \ell}}(f)$  can be written in the following block form:
\begin{align}\label{block-form-T-J}
T_{J_{N, \ell}}(f) = \left[
\begin{array}{cc}
T_{N}(f) &  \Lambda_{N,\ell}(f)
\\
\Lambda_{N,\ell}(f)^* & T_{N}(f)
\end{array}
 \right].
\end{align}
\end{proposition}

\begin{proposition}\label{prop-ratio}
For any integers $N\ge 1, \ell \ge 1$ and any $\epsilon, \epsilon' \in \{0, 1\}^N$, we have
\begin{align}\label{R-det-H}
R_{N,\ell}^{f}(\epsilon,\epsilon^{\prime}) = \det \Big( I - H_{N, \ell}^f( \epsilon, \epsilon')\Big),
\end{align}
where $H_{N, \ell}^f( \epsilon, \epsilon')$ is defined by
\begin{align}\label{def-H-N}
H^{f}_{N,\ell}(\epsilon,\epsilon^{\prime}):=\Big[T_{N}(f)+D_{N}(\epsilon^{\prime}-1)\Big]^{-1}\Lambda_{N,\ell}(f)^{*}\Big[T_{N}(f)+D_{N}(\epsilon-1)\Big]^{-1}\Lambda_{N,\ell}(f).
\end{align}
\end{proposition}

The following elementary identity for determinants will be used.
\begin{lemma}\label{lem-ABCD}
Let $A, B, C, D$ be $N\times N$ matrices such that  both $A$ and $D$ are invertible. Then
\[
\frac{\det
\left[
\begin{array}{cc}
A& B
\\
C & D
\end{array}
\right]
}{\det ( A ) \det (D)}   = \det \Big( I - D^{-1} C A^{-1}  B \Big).
\]
\end{lemma}

\begin{proof}[Proof of Proposition \ref{prop-ratio}]
Using the notation in Proposition \ref{prop-ST}, we have
\begin{align*}
&D_{2N}(2\epsilon-1,2\epsilon^{\prime}-1)T_{J_{N, \ell}}(f)+D_{2N}(1-\epsilon,1-\epsilon^{\prime}) = \left[
\begin{array}{cc}
A& B
\\
C & D
\end{array}
\right]
\end{align*}
where
\begin{align*}
A & :=D_{N}(2\epsilon-1)T_{N}(f)+D_{N}(1-\epsilon),
\\
D & :=D_{N}(2\epsilon'-1)T_{N}(f)+D_{N}(1-\epsilon'),
\\
B &=  D_N(2\epsilon-1)   \Lambda_{N,\ell}(f)
\\
 C &= D_N(2\epsilon^{\prime}-1)   \Lambda_{N,\ell}(f)^*
\end{align*}
Therefore, by the equalities  \eqref{DPP-f-C}, \eqref{DPP-f-C-bis} and Lemma \ref{lem-ABCD}, we obtain the desired equality \eqref{R-det-H}.  Note that here we used the identities $D_N(2\epsilon-1)^2 = D_N(2 \epsilon' -1)^2 = I$ to get
\begin{align*}
H^f_{N, \ell} (\epsilon, \epsilon') = D^{-1} C A^{-1}  B.
\end{align*}
\end{proof}

\begin{corollary}\label{Simon-Theorem}
For any integers $N\ge 1, \ell \ge 1$ and any $\epsilon, \epsilon' \in \{0, 1\}^N$, we have
\[\left|R_{N,\ell}^{f}(\epsilon,\epsilon^{\prime})-1\right|\leq \| H^{f}_{N,\ell}(\epsilon,\epsilon^{\prime})\|_1 \cdot \exp(\|H^{f}_{N,\ell}(\epsilon,\epsilon^{\prime})\|_1 +1).\]
\end{corollary}

\begin{proof}
It follows directly from Simon \cite[Theorem 6.5]{Simon-Adv77}.
\end{proof}

The following elementary result is classical.
\begin{lemma}\label{lem-inverse}
For a positive number $\tau \in (0, 1)$,  let $B$ be a bounded operator  on a Hilbert space $\mathcal{H}$ with $\|B\|\leq 1-\tau$. Then for each unitary operator $U$ on $\mathcal{H}$, the operator $U+B$ is invertible and $\|(U+B)^{-1}\|\leq \frac{1}{\tau}. $
\end{lemma}

\begin{lemma}\label{HS-Norm}
Assume that  $ \tau \le f \le 1-\tau$ for some $ \tau\in(0,1) $. Then for any integers $N, \ell \ge 1$ and any  $\epsilon, \epsilon' \in\{0, 1\}^N$,  we have
\[\| H^{f}_{N,\ell}(\epsilon,\epsilon^{\prime})\|_1 \leq \frac{1}{\tau^2}\|\Lambda_{N,\ell}(f)\|_{HS}^2,\]
where  $\Lambda_{N, \ell}(f)$ is defined in \eqref{def-lambda-N} and  $H_{N, \ell}^f(\epsilon, \epsilon')$ is defined in \eqref{def-H-N}.
\end{lemma}

\begin{proof}
Recall the definition  \eqref{def-T-phi} for $T(\phi)$. Note that $T(\phi)$ is unitary congugate to the operator of multiplication by $\phi$ on $L^2(\T)$ and in particular $\| T(\phi) \| = \| \phi\|_\infty$.

Write
\[
T_N(f) + D_N(\epsilon-1) =  \frac{1}{2}\Big( T_N(2f - 1) + D_N( 2\epsilon - 1) \Big) .
\]
Note that the assumption   $\tau \leq f\leq 1-\tau$ implies that  $ \| 2f -1\|_\infty \le 1 - 2 \tau$ and hence
\begin{align}\label{sup-norm-g}
\|T_N(2f-1)\|\le \| T(2f-1)\|\le 1 - 2 \tau.
\end{align}
Since $D_{N}(2 \epsilon-1)$ is a unitary matrix, by Lemma \ref{lem-inverse} and \eqref{sup-norm-g}, we have
\begin{align}\label{norm-inv-es}
\left\| \Big[   T_N(f) + D_N( \epsilon - 1)\Big]^{-1} \right\| = 2 \left\| \Big[   T_N(2f - 1) + D_N( 2\epsilon - 1)\Big]^{-1} \right\|\leq \frac{1}{\tau}.
\end{align}
The same inequality holds if we replace $\epsilon$ by $\epsilon'$ in \eqref{norm-inv-es}.
Thus  by \eqref{p-norm}, we have
\[
\| H^{f}_{N,\ell}(\epsilon,\epsilon^{\prime})\|_1 \leq \frac{1}{\tau}\Big\|\Lambda_{N,\ell}(f)^{*}\Big[T_{N}(f)+D_{N}(\epsilon-1)\Big]^{-1}\Lambda_{N,\ell}(f)\Big\|_1.
\]
Since $T_N(f) + D_N(\epsilon-1)$ is self-adjoint, by \eqref{self-bound},   the inequality \eqref{norm-inv-es} is equivalent to
\[
-\frac{1}{\tau}I\leq  \Big[T_{N}(f)+D_{N}(\epsilon-1)\Big]^{-1}\leq  \frac{1}{\tau}I.
\]
Hence, by \eqref{op-geq}, we get
$$-\frac{1}{\tau}\Lambda_{N,\ell}(f)^*\Lambda_{N,\ell}^{f}\leq  \Lambda_{N,\ell}(f)^{*}\Big[T_{N}(f)+D_{N}(\epsilon-1)\Big]^{-1}\Lambda_{N,\ell}(f) \leq \frac{1}{\tau}\Lambda_{N,\ell} (f)^*\Lambda_{N,\ell}^{f}.$$
Therefore, by Lemma  \ref{1-norm}, we obtain
$$\| H^{f}_{N,\ell}(\epsilon,\epsilon^{\prime})\|_1 \leq\frac{1}{\tau^2}\|\Lambda_{N,\ell} (f)^*\Lambda_{N,\ell} (f)\|_1=\frac{1}{\tau^2}\|\Lambda_{N,\ell}^{f}\|_{HS}^2.$$
\end{proof}
\begin{lemma}\label{go-to-0}
If $f\in H^{\frac{1}{2}}$, then for any integer $N\ge 1$,  we have
\[
\|\Lambda_{N,\ell} (f) \|_{HS}^2 \ge \sum_{k=1}^{N}k|\widehat{f}(\ell + k)|^2
\]
and
\[
\sup_{N\ge 1}\|\Lambda_{N,\ell} (f) \|_{HS}^2 \le \sum_{k=\ell+1}^{\infty}k|\widehat{f}(k)|^2.
\]
\end{lemma}
\begin{proof}
By definition,  for any integer $N \ge 1$,  we have
\begin{align*}
\|\Lambda_{N,\ell}^f\|^2_{HS}&= \sum_{k=1}^{N}k|\widehat{f}(\ell+k)|^2+\sum_{k=N+1}^{2N-1}(2N-k)|\widehat{f}(\ell+k)|^2.
\end{align*}
Thus the first inequality of the lemma is proved.  On the other hand, observing that  for any integer $N\ge 1$, we have $2N-k< k$ if $k\geq N+1$, hence
\begin{align*}
\|\Lambda_{N,\ell}^f\|^2_{HS}&\leq \sum_{k=1}^{N}k|\widehat{f}(\ell+k)|^2+\sum_{k=N+1}^{2N-1}k|\widehat{f}(\ell+k)|^2\\
                             &=\sum_{k=1}^{2N-1}k|\widehat{f}(\ell+k)|^2
                             \leq\sum_{k=1}^{2N-1}(\ell+k)|\widehat{f}(\ell+k)|^2=\sum_{k=\ell+1}^{\infty}k|\widehat{f}(k)|^2.
\end{align*}
This completes the proof of the lemma.
\end{proof}

\begin{lemma}\label{leqID}
For any integers $N,  \ell\ge 1$, we have
\begin{align}\label{T-Lambda-contract}
0 \le T_{N}(f)^{-1/2}\Lambda_{N,\ell}(f)^*T_{N}(f)^{-1}\Lambda_{N,\ell}(f ) T_{N}(f)^{-1/2}\leq I.
\end{align}
\end{lemma}
\begin{proof}
By Lemma \ref{lem-pd}, \eqref{T-Lambda-contract} follows from the block form \eqref{block-form-T-J} for the matrix $T_{J_{N, \ell}}(f)$ and the fact $T_{J_{N, \ell}}(f)$ is  non-negative definite.
\end{proof}

Now we are ready to prove Theorem \ref{psi-mixing}
\begin{proof}[Proof of Theorem \ref{psi-mixing}]

  For any integer $N \ge 1$, take  the particular word $\epsilon^*= (1, 1, \cdots, 1)  \in \{0, 1\}^N$ with constant coefficients $1$. Then by the definition of $\psi_{\mu_f}(\ell)$ and \eqref{R-det-H},  we have
\begin{align}\label{psi-f-low}
\begin{split}
\psi_{\mu_f}(\ell)& \ge \left|\frac{\mu_f([\epsilon^*] \cap \sigma^{-(N+\ell)} [\epsilon^*])}{\mu_f([\epsilon^*] ) \mu_f([\epsilon^*]) }-1\right|  \\
 &=\left|   \det\left( I- T_{N}(f)^{-1}\Lambda_{N,\ell}(f)^{*}T_{N}(f)^{-1}\Lambda_{N,\ell}(f) \right)-1\right|
\\
& \ge 1- \det\left( I- T_{N}(f)^{-1}\Lambda_{N,\ell}(f)^{*}T_{N}(f)^{-1}\Lambda_{N,\ell}(f) \right).
\end{split}
\end{align}
Since $0\le f\leq 1$ and $f \not \equiv 0, f \not \equiv 1$,  we have $0 \le T_{N}(f)\leq I$ and by Lemma \ref{lem-Lyons},  $T_N(f)$ is invertible. In particular, we have
\[T_{N}(f)^{-1}\geq I.\]
It follows that
\begin{align}\label{erase-to-I}
\Lambda_{N}(f)^{-1/2}\Lambda_{N,\ell}(f)^{*}T_{N}(f)^{-1}\Lambda_{N,\ell} (f) T_{N}(f)^{-1/2} \ge \Lambda_{N}(f)^{-1/2}\Lambda_{N,\ell}(f)^{*}\Lambda_{N,\ell} (f) T_{N}(f)^{-1/2}
\end{align}
and
\begin{align}\label{erase-to-I-bis}
\Lambda_{N, \ell}(f) T_N(f)^{-1} \Lambda_{N, \ell}(f)^* \ge \Lambda_{N, \ell}(f)  \Lambda_{N, \ell}(f)^*.
\end{align}
By Lemma \ref{leqID} and  \eqref{syl-id},  \eqref{det-exp}, \eqref{erase-to-I} and then \eqref{erase-to-I-bis}, we obtain
\begin{align}\label{up-det-HS}
\begin{split}
&\det\left( I- T_{N}(f)^{-1}\Lambda_{N,\ell}(f)^{*}T_{N}(f)^{-1}\Lambda_{N,\ell}(f) \right)\\
=&\det\left( I-T_{N}(f)^{-1/2}\Lambda_{N,\ell}(f)^{*}T_{N}(f)^{-1}\Lambda_{N,\ell} (f) T_{N}(f)^{-1/2} \right)
\\
\le &\exp\left[-\tr\Big(T_{N}(f)^{-1/2}\Lambda_{N,\ell}(f)^{*}T_{N}(f)^{-1}\Lambda_{N,\ell}(f) T_{N}(f)^{-1/2}\Big)\right]
\\
\le& \exp\left[-\tr\Big(T_{N}(f)^{-1/2}\Lambda_{N,\ell}(f)^{*}\Lambda_{N,\ell}(f) T_{N}(f)^{-1/2}\Big)\right]
\\
= & \exp\left[-\tr\Big(  \Lambda_{N,\ell}(f)   T_{N}(f)^{-1}\Lambda_{N,\ell}(f)^{*} \Big)\right]
\\
\le &  \exp\left[-\tr\Big(  \Lambda_{N,\ell}(f)   \Lambda_{N,\ell}(f)^{*} \Big)\right]  =  \exp\left(- \|\Lambda_{N,\ell}(f)\|_{HS}^2\right).
\end{split}
\end{align}
Therefore, by \eqref{psi-f-low}, \eqref{up-det-HS} and Lemma \ref{go-to-0}, we have
 \[\psi_{\mu_f}(\ell)\geq 1-\exp\left( -\sum_{k=1}^{N}k|\widehat{f}(\ell+k)|^2 \right), \quad \forall N\geq 1.\]
Since $N$ is arbitrary, we obtain
 \[\psi_{\mu_f}(\ell)\geq 1-\exp\left( -\sum_{k=1}^{\infty}k|\widehat{f}(\ell+k)|^2 \right).\]
Since $k \ge (k + \ell) / (\ell  + 1)$ for any integer $k \ge 1$, we have
\[
\sum_{k=1}^{\infty}k|\widehat{f}(\ell+k)|^2   \ge  \frac{1}{\ell + 1} \sum_{k = 1}^\infty (k + \ell) |\widehat{f}(\ell+k)|^2  =  \frac{1}{\ell + 1} \sum_{k =  \ell + 1}^\infty k |\widehat{f}(k)|^2
\]
and thus
\[\psi_{\mu_f}(\ell)\geq 1-\exp\left( - \frac{1}{\ell + 1}\sum_{k= \ell + 1}^{\infty}k|\widehat{f}(k)|^2 \right).\]
This is exactly the desired inequality \eqref{lower-bdd}.

Now if $\mu_f$ is $\psi$-mixing, then the condition \[\lim_{\ell \to\infty}\psi_{\mu_f}(\ell) = 0\] implies in particular that $\sum_{k = 1}^\infty k |\widehat{f}(k)|^2 <\infty$, which combined with the assumption that $f$ is real-valued, implies that $f\in H^{1/2}(\T)$.

Finally, we show that  if $f\in H^{\frac{1}{2}}(\T)$ and $\tau \le f \le 1- \tau$ then $\mu_f$ is $\psi$-mixing. By Lemma \ref{lem-sup}, Corollary \ref{Simon-Theorem}, Lemma \ref{HS-Norm} and then Lemma \ref{go-to-0},  for any $\ell \ge 1$,  we have
\begin{align*}
\psi_{\mu_f}(\ell)  & \le \sup_{N\ge 1} \Big(  \| H^{f}_{N,\ell}(\epsilon,\epsilon^{\prime})\|_1 \cdot       \exp(\|H^{f}_{N,\ell}(\epsilon,\epsilon^{\prime})\|_1 +1)       \Big)
\\
& \le \sup_{N\ge 1} \Big(  \frac{1}{\tau^2}\|\Lambda_{N,\ell}(f)\|_{HS}^2 \cdot \exp\Big( \frac{1}{\tau^2}\|\Lambda_{N,\ell}(f)\|_{HS}^2 +1\Big) \Big)
\\
& \le  \frac{1}{\tau^2} \Big(\sum_{k=\ell+1}^{\infty}k|\widehat{f}(k)|^2  \Big) \exp\Big(   \frac{1}{\tau^2} \sum_{k=\ell+1}^{\infty}k|\widehat{f}(k)|^2  + 1  \Big).
\end{align*}
This is exactly the inequality \eqref{upper-bound} and the $\psi$-mixing property of $\mu_f$ now follows from the assumption $f\in H^{1/2}(\T)$.
\end{proof}

\section{$L^{q}$-dimensions of stationary determinantal measures}
Recall the definition \eqref{def-q-dim} for the $L^q$-dimension of a probability measure on a metric space.
In this section, we investigate the $L^q$-dimension of the stationary determinantal measures $\mu_f^{+}$ on the metric space $(\{0, 1\}^\N, d)$, where $d$ is defined by \eqref{def-metric}.

  In this section, by slightly abusing the notation $[\epsilon] \subset \{0, 1\}^\Z$ introduced in \eqref{def-cyl},  for any $\epsilon \in \{0, 1\}^N$, by $[\epsilon]$, we mean the corresponding cylinder set in $\{0, 1\}^\N$ defined by
\[
[\epsilon] := \Big\{x\in \{0, 1\}^\N \Big|x_0x_1\cdots x_{N-1}  = \epsilon \Big\}.
\]

\subsection{Existence of $L^q$-dimensions}

Recall the definition \eqref{def-S-N} of $S_N^{(q)}(\mu_f^{+})$. By the definition \eqref{def-metric} of the metric $d$,  for any $x\in \{0, 1\}^\N$, we have $B(x, 2^{-N}) = [x_0x_1\cdots x_{N-1}]$. Therefore,
\begin{align}\label{comp-S-N}
S_N^{(q)}(\mu_f^{+}) & =\int_{\{0, 1\}^\N} \mu_f^{+}(B(x,2^{-N}))^{q-1} d\mu_f^{+}(x)
= \sum_{\epsilon \in \{0,1\}^N} \mu_f^{+}([\epsilon])^{q}.
\end{align}
By \eqref{DPP-f-C}, for any $\epsilon \in \{0,1\}^N$, we have
\begin{align}\label{plus-cyl-w}
\mu_f^{+}([\epsilon])=\det\Big(D_{N}(2\epsilon-1)T_{N}(f)+D_{N}(1-\epsilon)\Big).
\end{align}
The following simple lemma will be useful in our computation:
\begin{lemma}\label{lem-matrix-id}
For any $\epsilon \in \{0, 1\}^N$, we have
\begin{align}\label{simple-matrix-id}
D_{N}(2\epsilon-1)T_{N}(f)+D_{N}(1-\epsilon)= \frac{1}{2} \Big( D_N(2 \epsilon-1) T_N(2f-1) +  I\Big),
\end{align}
where $I$ stands for the $N\times N$ identity matrix.
\end{lemma}

For simplifying our notation, in what follows, we denote
\begin{align}\label{def-theta-g}
\theta: = 2 \epsilon -1 \in \{\pm 1\}^N, \quad g: = 2 f -1.
\end{align}
By \eqref{comp-S-N},  \eqref{plus-cyl-w} and \eqref{simple-matrix-id}, using the notation \eqref{def-theta-g}, we have

\begin{align}\label{S-exp}
\begin{split}
 S_N^{(q)}(\mu_f^{+}) & = \sum_{\theta\in \{\pm 1\}^N} \frac{1}{2^{qN}} \dt^{q}(I +D_N(\theta)T_{N}(g))
\\
& =  \frac{1}{2^{(q-1)N}}  \sum_{\theta \in \{\pm 1\}^N }   \frac{1}{2^N} \dt^{q}(I +D_N(\theta)T_{N}(g)) \\  &=\frac{1}{2^{(q-1)N}} \E_{\theta}[\dt^{q}(I +D_N(\theta)T_{N}(g))]  = \frac{1}{2^{(q-1)N}} \Sigma_{N}^{(q)} (g),
\end{split}
\end{align}
where $\E_\theta$ means the expectation on $\theta$ with respect to the normalized Haar measure of the finite  group $\{\pm 1\}^N$ and
\begin{align}\label{def-sigma-q}
\Sigma_{N}^{(q)} (g):= \E_{\theta}[\dt^{q}(I +D_N(\theta)T_{N}(g))].
\end{align}

Now we are ready to prove Proposition \ref{prop-sub-m}.
\begin{proof}[Proof of Proposition \ref{prop-sub-m}]
Note that $\mu_f^{+}([\epsilon]) = \mu_{1-f}^{+} ([1- \epsilon])$ for any $\epsilon \in \{0, 1\}^N$. Then by the definition of $S_N^{(q)}(\mu_f^{+})$, we have $S_N^{(q)}(\mu_f^{+}) = S_N^{(q)}(\mu_{1-f}^{+})$.
Therefore,  we only need to prove Proposition \ref{prop-sub-m} under the assumption that $1/2 \le f \le 1$.

The assumption $1/2 \le f \le 1$ implies that  $g = 2f -1 \ge 0$.
By \eqref{S-exp}, it suffices to prove for any integers $M, N \ge 1$, we have
\begin{align}\label{des-sub-m}
\Sigma_{M+N}^{(q)}(g) \leq \Sigma_{M}^{(q)} (g) \cdot \Sigma_{N}^{(q)} (g) .
\end{align}
Using  \eqref{det1+L}, we have
\[
\dt(I +D_N(\theta)T_{N}(g))= \sum_{J\subset[N]}\det  \Big[\Big( D_N(\theta)\cdot  T_{J}(g) \Big)_J\Big]= \sum_{J\subset[N]}\det D_{J}(\theta)\cdot \det T_{J}(g).
\]
For any subset $J \subset [N]$, set
\begin{align}\label{def-Walsh} a_J:= \det T_{J}(g) \an w_{J}: =\det D_{J}(\theta)= \prod_{j\in J} \theta_j.
\end{align}
Then
 \begin{align}\label{sigma-N-q}
\Sigma_N^{(q)} (g)  = \E_{\theta}[\dt^q(I+D_{N}(\theta)T_{N}(g))]=\sum_{J_1, \cdots, J_q \subset [N]} \E_{\theta}[w_{J_1}\cdots w_{J_q}] a_{J_1}\cdots a_{J_q}.
\end{align}

Note that the assumption $g\ge 0$ implies that for any finite subset $J\subset \N$, the matrix $T_J(g)$ is non-negative definite. Therefore $a_J \ge 0$. Moreover, using Fischer's inequality \eqref{Fischer-inq}, for any pair of disjoint subsets $J_1, J_2 \subset \N$, we have
\begin{align}\label{pos-a-J}
 0 \le a_{J_1\sqcup J_2} \leq a_{J_1}\cdot a_{J_2}.
\end{align}
 Note also  that
 \[\E[w_{J_1} \cdots w_{J_{q}}]  = \left\{
\begin{array}{cl}
1 &  \text{if $\sum_{k=1}^q \delta(i\in J_k)$ is even for all $i \in \N$}
\vspace{2mm}
\\
0 & \text{otherwise}
\end{array}
\right..
\]
In particular, for any finite subsets $J_1, \cdots, J_q \in \N$,  we have
\begin{align}\label{pos-c-J}
c(J_1,J_2,\cdots,J_q) : = \E[w_{J_1} \cdots w_{J_{q}}] \ge 0
\end{align}
Note that any subset $J\subset [M+N]$ can be written in a  disjoint union
\[
J= \Big(J\cap [M] \Big) \sqcup \Big(J\cap ([N] +M)\Big),
\]
where $[N]+M: = \{1 + M, 2 + M, \cdots, N+M\}$. Thus by using \eqref{pos-a-J} and \eqref{pos-c-J},  we have
 \begin{align*}
\Sigma_{M+N}^{(q)} (g)  &= \sum_{J_1, \cdots, J_q \subset [M+N]}c(J_1,\cdots,J_q)a_{J_1}\cdots a_{J_q} \\
  &=\sum_{ \substack{J_1^\prime, \cdots, J_{q}^{\prime} \subset [M]\\J_{1}^{''}, \cdots, J_{q}^{''} \subset [N]+M} }c(J_1^{\prime}\sqcup J_1^{''},\cdots,J_q^{\prime}\sqcup J_q^{''} )a_{J_1^{\prime}\sqcup J_1^{''}}\cdots a_{J_q^{\prime}\sqcup J_q^{''}} \\
  &\leq \sum_{ \substack{J_1^\prime, \cdots, J_{q}^{\prime} \subset [M]\\J_{1}^{''}, \cdots, J_{q}^{''} \subset [N]+M} }c(J_1^{\prime}\sqcup J_1^{''},\cdots,J_q^{\prime}\sqcup J_q^{''} )a_{J_1^{\prime}}\cdot a_{ J_1^{''}}\cdots a_{J_q^{\prime}} a_{J_q^{''}}.
 \end{align*}
Since the two subsets $J_1' \cup \cdots \cup J_q'$ and $J_1'' \cup \cdots \cup J_q''$ are disjoint,  by the definitions of the functions  $w_J$, the two random variables $\prod_{k=1}^q w_{J_k'}$ and $\prod_{k=1}^q w_{J_k''}$ are independent and
\[
\prod_{k=1}^q w_{J_k' \sqcup J_k''} = \prod_{k=1}^q w_{J_k'} \cdot \prod_{k=1}^q w_{J_k''}
\]
 Hence for any $J_1^\prime, \cdots, J_{q}^{\prime} \subset [M]$ and any $J_{1}^{''}, \cdots, J_{q}^{''} \subset [N]+M$,  we have
\begin{align*}
c(J_1^{\prime}\sqcup J_1^{''},\cdots,J_q^{\prime}\sqcup J_q^{''} )& =\E_{\theta}\Big[  \prod_{k=1}^q w_{J_k' \sqcup J_k''} \Big]  =\E_{\theta}\Big[ \prod_{k=1}^q w_{J_k'} \cdot \prod_{k=1}^q w_{J_k''} \Big]
\\
& = \E_{\theta}\Big[ \prod_{k=1}^q w_{J_k'}\Big] \cdot \E_\theta \Big[\prod_{k=1}^q w_{J_k''} \Big]  = c(J_1', \cdots, J_q') \cdot c(J_1'', \cdots, J_q'')
\\
&  = c(J_1^{\prime}, \cdots, J_q^{\prime} ) c(J_1^{''}-M,\cdots, J_q^{''}-M ),
\end{align*}
where $J_k''-M$ is the shifted set of $J_k''$ defined by $J_k''-M= \{j-M: j \in J_k''\}$. Observing that  $(J_1'', \cdots, J_q'')$ ranges over all $q$-tuples of subsets of $[N]+M$ if and only if $(J_1''-M, \cdots, J_q''-M)$ ranges over all $q$-tuples of subsets of $[N]$.  Hence, by using the identities
\[
a_{J_k''} = a_{J_k''-M},
\]
 we obtain
\begin{align*}
&     \Sigma_{M+N}^{(q)}(g) \le \sum_{ \substack{J_1^\prime, \cdots, J_{q}^{\prime} \subset [M]\\J_{1}^{''}, \cdots, J_{q}^{''} \subset [N]+M} }c(J_1^{\prime}\sqcup J_1^{''},\cdots,J_q^{\prime}\sqcup J_q^{''}
 )a_{J_1^{\prime}} a_{ J_1^{''}}\cdots a_{J_q^{\prime}} a_{J_q^{''}}
\\
& = \sum_{ \substack{J_1^\prime, \cdots, J_{q}^{\prime} \subset [M]\\J_{1}^{''}, \cdots, J_{q}^{''} \subset [N] + M} }c(J_1^{\prime}, \cdots, J_q^{\prime} ) c(J_1^{''}-M,\cdots, J_q^{''}-M )\cdot a_{ J_1^{''}}\cdots a_{J_q^{\prime}} \cdot  a_{J_1^{''}-M}\cdots a_{J_q^{''}-M}
\\
& = \sum_{ \substack{J_1^\prime, \cdots, J_{q}^{\prime} \subset [M]\\ \widetilde{J_{1}}, \cdots, \widetilde{J_{q}} \subset [N]} }c(J_1^{\prime}, \cdots, J_q^{\prime} ) c(\widetilde{J_1},\cdots, \widetilde{J_q})\cdot a_{ J_1^{''}}\cdots a_{J_q^{\prime}} \cdot  a_{\widetilde{J_1}}\cdots a_{\widetilde{J_q}}
=\Sigma_{M}^{(q)} (g) \cdot \Sigma_{N}^{(q)}(g).
 \end{align*}
We thus obtain  the desired inequality  \eqref{des-sub-m} and  complete the proof of Proposition \ref{prop-sub-m}.
\end{proof}

\begin{proof}[Proof of Theorem \ref{thm-CD}]
Note that for $r = 2^{-N}$, we have
\begin{align}\label{log-ratio}
\frac{ \log \int_{\{0, 1\}^\N} \mu_f^{+}(B(x,2^{-N}))^{q-1} d\mu_f^{+}(x)}{\log (2^{-N})} =  \frac{\log S_N^{(q)}(\mu_f^{+})}{ N \log (1/2)}.
\end{align}
Thus Theorem \ref{thm-CD} follows from Proposition  \ref{prop-sub-m}
and Fekete's Subadditive Lemma.
\end{proof}

\subsection{Estimations of the correlation dimension}
Recall the definition \eqref{def-sigma-q} for $\Sigma_N^{(q)}(g)$.

\begin{lemma}\label{lem-q-2}
When $q = 2$, we have
\[\Sigma_N^{(2)}(g):=\sum_{J\subset [N]}\dt^2 T_J(g).\]
\end{lemma}
\begin{proof}
This follows immediately from the equality \eqref{sigma-N-q} and the elementary fact that
\[
\E[w_{J_1} w_{J_2}] = \delta_{J_1 = J_2}, \quad \forall J_1, J_2 \subset [N].
\]
\end{proof}

\begin{theorem}[Szeg\"{o}'s First Theorem \cite{GS-Toeplitz,Szego-OP}]Let $\phi$ be a non-negative Lebesgue integrable function on the unit circle. Then
$$\lim_{N\to \infty}\frac{ \log \det T_{N}(\phi)}{N}=\int_{0}^{1}\log \phi(e^{i2 \pi t})dt.$$
\end{theorem}

\begin{lemma}\label{J-square}
Let  $\phi: \T\rightarrow \R$ be a real-valued bounded Borel function. Then for any subset $J\subset \Z$, we have
$$T_{J}(\phi)^2\leq T_{J}(\phi^2).$$
\end{lemma}

\begin{proof}
For any $J\subset \Z$, the space $\ell^2(J)$ is naturally identified with a subspace of $\ell^2(\Z)$. Let $P_J$ denote  $P_{J}$ is the orthogonal projection from $\ell^{2}(\Z)$ onto $\ell^2(J)$, then we have
\[T_J(\phi)=P_{J}T(\phi)P_{J} \an T_J(\phi^2) = P_J T(\phi^2) P_J. \]
Note that the assumption that $\phi$ is real-valued implies that $T(\phi)$ is self-adjoint.  The elementary operator inequality  $P_J \le I$, where $I$ is the identity operator on $\ell^2(\Z)$,  combined with   \eqref{op-geq}, implies that
\begin{align*}
T_J(\phi)^2 &  = P_J T(\phi) P_J P_J T(\phi) P_J = [ T(\phi) P_J]^* P_J [T(\phi) P_J]
\\
& \le [ T(\phi) P_J]^*  [T(\phi) P_J] = P_J T(\phi)^2 P_J.
\end{align*}
A direct computation shows that
\[
\widehat{\phi^2}(i - j) =  \sum_{k\in \Z} \widehat{\phi}(i -k) \widehat{\phi}(k -j), \quad \text{for all $i, j \in \Z$},
\]
hence we have $T(\phi^2)=T(\phi)^2$. The desired inequality now follows immediately:
\[
T_J(\phi)^2 \le P_J T(\phi)^2 P_J    =  P_J T(\phi^2) P_J = T_J(\phi^2).
\]
\end{proof}

\begin{proposition}\label{prop-Lower}
Let $f: \T \to [0,1]$ be a Borel function. Then we have
\begin{align}\label{2-dim-low}
\underline{\dim}_2\mu_f^{+}\geq  \frac{1}{\log 2} \int_{0}^{1}\log\frac{2}{1+ ( 2f(e^{i 2 \pi t})-1)^2}dt.
\end{align}
\end{proposition}
\begin{proof}
Set $g = 2f -1$. By Lemma \ref{lem-q-2} and Lemma  \ref{J-square}, we obtain
\begin{align*}
S_N^{(2)}(\mu_f^{+}) &= 2^{-N}\sum_{J\in[N]}\dt^2 T_J(g)
\leq  2^{-N}\sum_{J\in[N]}\det T_J(g^2) \\
    &=  2^{-N}\det (I+ T_N(g^2))=  \det \left( T_N(\frac{1+g^2}{2})\right).
\end{align*}
Hence, by \eqref{log-ratio}, we have
\begin{align*}
\underline{\dim}_2\mu_f^{+} =\liminf_{N\to \infty}{\log S_{N}^{(2)}(\mu_f^{+}) \over -N \log 2} \geq  \liminf_{N\to \infty}{\log \det \left( T_N(\frac{1+g^2}{2})\right) \over -N \log 2}.\\
\end{align*}
Now the inequality \eqref{2-dim-low} follows immediately by applying  Szeg\"{o}'s First Theorem to the function $(1 + g^2)/2$.
\end{proof}

\begin{proposition}\label{prop-upper}
Let $f: \T \to [0,1]$ be a Borel function. Then we have
\[\overline{\dim}_2\mu_f^{+}\leq 1- \frac{1}{ \log 2} \int_{0}^{1}\log \left( \frac{[1+\beta (2f(e^{i 2 \pi t}) -1)]^2}{1+\beta^2}\right) d t,  \quad \forall \beta \in [-1,1]. \]
\end{proposition}

\begin{lemma}\label{Upper}
For $\beta\in \R$,  we have
\[\Sigma_{N}^{(2)}(g) \geq \frac{\dt^2(T_{N}(1+\beta g))}{(1+\beta^2)^N}= \dt^2\Big(T_{N}(\frac{1+\beta g}{\sqrt{1+\beta^2}})\Big).\]
\end{lemma}
\begin{proof}
By Cauchy-Schwartz inequality, we have
\begin{align*}
 \dt^2(T_{N}(1+\beta g)) &  = \Big(  \sum_{J\in [N]} \det (T_{J}(g))\cdot \beta^{|J|}\Big)^2 \leq \Big(\sum _{J\in [N]} \dt^2 T_{J}(g)\Big) \Big(\sum _{J\in[T]} \beta^{2|J|}\Big) \\
&= \Big(\sum _{J\in [N]} \dt^2 T_{J}(g)\Big) \det\left(I+\beta^2 I\right)
\\
& = \Sigma_N^{(2)}(g)  \cdot ( 1 + \beta^2)^N.
\end{align*}
This completes the proof of the lemma.
\end{proof}

\begin{proof}[Proof of Proposition \ref{prop-upper}]
Set $g = 2 f -1$.  By \eqref{log-ratio}, we have
\begin{align*}
\overline{\dim}_2\mu_f^{+}=\limsup_{N\to \infty}{\log S_{N}^{(2)}(\mu_f^{+}) \over -N \log 2} =  \limsup_{N\to \infty}\left(1- {\log \Sigma_N^{(2)}(g) \over N \log 2} \right).
\end{align*}
By Lemma \ref{Upper}, we have
\[ \overline{\dim}_2\mu_f \leq 1-  \lim_{N\to  \infty}  \frac{ 2 \log \left(\dt\left(T_{N}(\frac{1+\beta g}{\sqrt{1+\beta^2}})\right) \right)}{N\log2}. \]
Since for $\beta \in [-1, 1]$, we have $1 + \beta g \ge 0$.  By Szeg\"{o}'s first Theorem,   we have
\[ \lim_{N\to  \infty}  \frac{ 2 \log \Big(\dt\Big(T_{N}(\frac{1+\beta g}{\sqrt{1+\beta^2}})\Big)\Big)}{N} = \int_{0}^{1}\log\frac{(1+\beta g(e^{i 2 \pi \theta}))^2} {1+\beta^2}d\theta. \]
Hence, we have
$$\overline{\dim}_2\mu_f^{+} \leq 1- \frac{1}{\log 2} \int_{0}^{1}\log\frac{(1+\beta g(e^{i 2 \pi \theta}))^2} {1+\beta^2}d \theta,  \quad \forall \beta \in [-1,1].$$
\end{proof}

\section{An application: Increasing rate of longest common substring}

Let $\mu$ be a shift-invariant probability measure on $\{0,1\}^\N$. For any $n \ge 1$ and any two sequences $x, y \in \{0,1\}^{\mathbb{N}}$,  we define the length of their longest common substring in their prefixes of length $n$ by
\[M_n(x,y)=\max\{m:x_{i+k}=y_{j+k}\textrm{ for $k=1,\dots,m$ and for some $1\leq i,j\leq n-m$}\}.\]
Recently, Barros, Liao and Rousseau \cite{BLRADV} showed that the correlation dimension of $\mu$ describes the increasing rate of  $M_n(x, y)$ for $\mu\otimes \mu$-typical pair of $(x, y) \in \{0,1\}^\N\times \{0,1\}^\N$.

\begin{theorem}[{\cite[Theorem 7]{BLRADV}}]\label{thm-BLR}
For $\mu\otimes\mu$-almost every $(x,y)\in  \{0,1\}^\N \times  \{0,1\}^\N$,
\[ \underset{n\rightarrow+\infty}{\overline\lim}\frac{\log M_n(x,y)}{\log n}\leq\frac{2/\log 2}{  \underline{\dim}_2\mu}.\]
Moreover, if the $\mu$ is $\psi$-mixing with $\psi(\ell)=O(\ell^{-a})$ for some $a>0$, then for $\mu\otimes\mu$-almost every $(x,y)\in \{0,1\}^\N\times \{0,1\}^\N$,
\[ \underset{n\rightarrow+\infty}{\underline\lim}\frac{\log M_n(x,y)}{\log n} \ge \frac{2/\log 2}{\overline{\dim}_2\mu}.\]
\end{theorem}

\begin{remark}\label{rem-sob}
For any $s>0$,  a function  $f: \T\rightarrow \C$ is said to be in the Sobolev space $H^{s}(\T)$ if
 \begin{align}\label{def-Sob-g}
\sum_{n=-\infty}^{+\infty}|n|^{2s}|\widehat{f}(n)|^{2}<  \infty.
\end{align}
In particular, if $s> 1/2$,  then \eqref{def-Sob-g} implies  (see, e.g., Moricz \cite[Lemma 1]{Moricz})
\[
\sum_{|n|> \ell} | n | |\widehat{f}(n)|^2 = O (\ell^{- (2s-1)}).
\]
\end{remark}

By Theorems \ref{thm-BLR}, \ref{psi-mixing}, \ref{thm-CD} and Remark \ref{rem-sob}, we have the following corollary.
\begin{corollary}
Assume that $f \in H^{1/2 + \varepsilon}(\T)$  for some $\varepsilon > 0$ such that either $1/2 \le f \le 1 - \tau$  or $\tau \le f \le 1/2$ for some $\tau>0$,  then for $\mu_f^{+}\otimes\mu_f^{+}$-almost every $(x, y)\in \{0, 1\}^\N \times \{0, 1\}^\N$, we have
\[ \underset{n\rightarrow+\infty}{\lim}\frac{\log M_n(x, y)}{\log n}=\frac{2/\log 2}{{\dim}_2\mu_f^{+}}.\]
\end{corollary}

\section*{Acknowledgements}
S. Fan is supported by NSFC 11971190 and  the Fundamental Research Funds for the Central Universities(CCNU19QN076). Y. Qiu is supported by grants NSFC Y7116335K1,  NSFC 11801547 and NSFC 11688101 of National Natural Science Foundation of China.

%
%
\bibliographystyle{plain}

\end{document}